\documentclass[10pt]{article}
\usepackage{hyperref}
\usepackage[alphabetic]{amsrefs}
\usepackage{amssymb}
\usepackage{amsfonts}
\usepackage{amsthm}
\usepackage{amsmath}
\usepackage{float}
\usepackage{bigints}

\newtheorem{Theorem}{Theorem}

\newtheorem{Definition}{Definition}
\newtheorem{Problem}{Problem}
\newtheorem{Corollary}{Corollary}

\usepackage[pdftex]{graphicx}
\graphicspath{{C:/Users/Sam/Desktop/PaperFigures/}}
\DeclareGraphicsExtensions{.pdf}

\begin{document}
\title{On Non-Standard Models of Peano Arithmetic and Tennenbaum's Theorem}
\author{Samuel Reid\thanks{Undergraduate student at the University of Calgary; \textsf{E-mail: smrei@ucalgary.ca}}}
\maketitle

\begin{abstract}
Throughout the course of mathematical history, generalizations of previously understood concepts and structures have led to the fruitful development of the hierarchy of number systems, non-euclidean geometry, and many other epochal phases in mathematical progress. In the study of formalized theories of arithmetic, it is only natural to consider the extension from the standard model of Peano arithmetic, $\langle \mathbb{N},+,\times,\leq,0,1 \rangle$, to non-standard models of arithmetic. The existence of non-standard models of Peano arithmetic provided motivation in the early $20^{th}$ century for a variety of questions in model theory regarding the classification of models up to isomorphism and the properties that non-standard models of Peano arithmetic have. This paper presents these questions and the necessary results to prove Tennenbaum's Theorem, which draws an explicit line between the properties of standard and non-standard models; namely, that no countable non-standard model of Peano arithmetic is recursive. These model-theoretic results have contributed to the foundational framework within which research programs developed by Skolem, Rosser, Tarski, Mostowski and others have flourished. While such foundational topics were crucial to active fields of research during the middle of the $20^{th}$ century, numerous open questions about models of arithmetic, and model theory in general, still remain pertinent to the realm of $21^{st}$ century mathematical discourse.
\end{abstract}

\pagebreak

\section{Historical and Philosophical Motivation}
Prior to delineating the terminological parameters for the present analysis of non-standard models of peano arithmetic, the historical and philosophical motivation behind the project must be rendered explicit. The important background for the development of non-standard models of mathematics, such as non-euclidean geometry, provide a coherent context to discuss the non-standard models of arithmetic developed during the $20^{th}$ century. The story of non-euclidean geometry stems from the Islamic Golden Age taking place from the $9^{th}$ to $12^{th}$ century\cite{Berggren}, in which much of the remaining knowledge from Babylonian, Egyptian, and Greek mathematics were compiled by Arab scholars. Translations of these writings subsequently spread to Europe during the Middle Ages and provided European mathematicians the foundation to work for centuries on contributing to the growing body of knowledge of mathematics. During the $18^{th}$ century, seemingly independently\cite{Gray}, J\'{a}nos Bolyai and Nikolai Lobachevsky considered non-standard models of geometry, such as Hyperbolic geometry, which challenged the geometry of Euclid accepted since antiquity. Non-euclidean geometry was further pursued by Bernhard Riemann\cite{Laugwitz} and led to the development of metric tensors and manifolds which provided the mathematical foundation for Albert Einstein's theory of General Relativity. Since the consideration of non-standard models of geometry were successful in providing a coherent framework for generalizing mathematics to such crowning intellectual achievements as General Relativity, surely the consideration of non-standard models of arithmetic will lead to important results in the foundations of mathematics as well.

At the time of Riemann and Einstein, research programs championed by Richard Dedekind, Bertrand Russell, David Hilbert, and others, led to the investigation of formalized axiomatic systems as a foundation for mathematics\cite{Ewald}. One such axiomatic system, known as Peano arithmetic, was proposed by the Italian mathematician Giuseppe Peano in his 1889, ``\textit{Arithmetices principia, nova methodo exposita}''. Peano arithmetic became the paradigmatic axiomatization of arithmetic, and also provided mathematicians with a concrete example of the practical expressiveness of results that can be proved directly from axioms by logic. Furthermore, major paradoxes in naive set theory, such as Russell's paradox, required reformulations of entire fields of mathematics; simultaneously, the philosophical positions of logicism, formalism, and intuitionism on the epistemology and ontology of mathematics began to develop\cite{Mancosu}. Yet throughout this burst of activity in the foundations of mathematics, arguably the two most important lines of research to emerge from this period were Zermelo-Fraenkel set theory and Peano arithmetic; and, starting with Cantor, the clarification of the concept of infinity as described by ordinals and the cardinality of sets\cite{Ferreiros}. This paper engages in an analysis of non-standard models of Peano arithmetic and the developments in model theory and set theory resulting from the research done during the early and mid $20^{th}$ century on foundations of mathematics. The results discussed include the existence of $2^{\aleph_{0}}$ non-standard models of Peano arithmetic, Overspill, and Tennenbaum's Theorem.

\section{Lexicon}
The first-order language $\mathsf{L}$ contains a truth-functionally complete set of boolean connectives, equality, existential and universal quantification, an endless supply of variables $v_{i}$, brackets, and non-logical symbols for constants, relations and functions. While there are numerous ways to configure this first-order language, the notion of a structure naturally formulates $\mathsf{L}$ for useful applications.\cite{Hodges}
\begin{Definition}
$\mathcal{M}$ is an $\mathsf{L}$-structure when $\mathcal{M}$ subsumes the following,
\begin{enumerate}
\item The set of all elements of $\mathcal{M}$, called the domain, dom$(\mathcal{M})$.
\item The set of elements of $\mathcal{M}$ which are constant symbols. For a constant $c$, the constant element named by $c$ is $c^\mathcal{M}$.
\item For each positive integer $n$, the set of $n$-ary relations on dom$(\mathcal{M})$. For a relation symbol $R$, the relation named by $R$ is $R^\mathcal{M}$.
\item For each positive integer $n$, the set of $n$-ary operations on dom($\mathcal{M})$ called functions $f: (\text{dom}(\mathcal{M}))^{n} \rightarrow \text{dom}(\mathcal{M})$. For a function symbol $F$, the function named by $F$ is $F^\mathcal{M}$.
\end{enumerate}
\end{Definition}
A particular type of $\mathsf{L}$-structure, called a model, gives - rather informally - the correct interpretation to a set of $\mathsf{L}$-sentences called an $\mathsf{L}$-theory. That is,
\begin{Definition}
Let $T$ be an $\mathsf{L}$-theory and let $\varphi \in T$. If $\mathcal{M} \vDash \varphi$ for all $\varphi \in T$, then $\mathcal{M}$ is a model for $T$, written $\mathcal{M} \vDash T$.
\end{Definition}
For discussing models of arithmetic, the first-order language of arithmetic $\mathsf{L_A}$ consists of the constant symbols $0$ and $1$, the binary relation symbol\footnote{The relation $\leq: \mathbb{N} \times \mathbb{N}$, called the total order relation, is that for all $a,b \in \mathbb{N}, a \leq b$ if and only if there exists some $c \in \mathbb{N}$ such that $a + c =b$.} $\leq$, and two binary function symbols $+:\mathbb{N}^2 \rightarrow \mathbb{N}$ and $\times:\mathbb{N}^2 \rightarrow \mathbb{N}$. The main interest of this paper, Peano arithmetic, is a particular $L_{A}$-theory defined as follows.\cite{Smith}
\begin{Definition}
(Peano Arithmetic): $\mathsf{PA}$ is an $\mathsf{L_A}$-theory with the axioms,
\begin{enumerate}
\item $\forall x (x+1 \neq 0)$.
\item $\forall x ((x+1=y+1 \Rightarrow x=y))$.
\item $\forall x (x + 0 = 0)$.
\item $\forall x \forall y (x + (y+1) = (x+y) +1)$.
\item $\forall x (x \times 0 = 0)$.
\item $\forall x \forall y (x \times (y+1) = (x \times y) + x)$.
\end{enumerate}
and every sentence that is an instance of the induction schema, $$\text{Ind}(\varphi) := (\varphi(0) \wedge \forall x(\varphi(x) \Rightarrow \varphi(x+1))) \Rightarrow \forall x \varphi(x)$$ where $\varphi(x)$ is an $L_{A}$-sentence.
\end{Definition}
\pagebreak
The crucial definition of a model for Peano arithmetic immediately follows.
\begin{Definition}
A model for Peano arithmetic $\mathcal{M} \vDash \mathsf{PA}$ is an $\mathsf{L_A}$-structure such that $\mathcal{M} \vDash \varphi$ for all $\varphi \in \mathsf{PA}$.
\end{Definition}
The signature of an $\mathsf{L}$-structure lists the set of functions, relations and constants of that structure. So, the $\mathsf{L_A}$-structure with domain $\{1,2,3,...\}$ is known as the standard model $\mathcal{N}$ with a signature given by $\langle \mathbb{N},+,\times,\leq,0,1 \rangle$.

In order to investigate the relationships between standard and non-standard models of $\mathsf{PA}$, basic definitions and results concerning $\mathsf{L}$-structures and models for $\mathsf{L}$-theories must be presented. Toward this end, two main questions that were important in the development of model theory in the early $20^{th}$ century are the following, given distinct $\mathsf{L}$-structures $\mathcal{M}$ and $\mathcal{K}$ and a theory $T$,
\begin{enumerate}
\item When can $\mathcal{M}$ and $\mathcal{K}$ be considered to model $T$ in the same way?
\item When can $\mathcal{M}$ and $\mathcal{K}$ be considered to model $T$ differently?
\end{enumerate}
The answer to the first question is that $\mathcal{M}$ and $\mathcal{K}$ model $T$ in the same way when a type of function $f: \mathcal{M} \rightarrow \mathcal{K}$ is an isomorphism. The answer to the second question is much deeper and will require the rest of the paper to answer. In order to make the first answer more precise, we show that the ``type of function'' referred to is in fact a homomorphism satisfying certain conditions.\cite{Hodges}
\begin{Definition}
Let $S$ be a signature and let $\mathcal{A}$ and $\mathcal{B}$ be $\mathsf{L}$-structures. A homomorphism $f: \mathcal{A} \rightarrow \mathcal{B}$, is a function $f$ from dom$(\mathcal{A})$ to dom$(\mathcal{B})$ such that,
\begin{enumerate}
\item For each constant $c$ of $S$, $f(c^\mathcal{A})=c^\mathcal{B}$.
\item For each $n>0$, $n$-ary relation symbol $R$ of $S$ and n-tuple $\bar{a} \in \mathcal{A}$, if $\bar{a} \in R^\mathcal{A}$ then $f \bar{a} \in R^\mathcal{B}$.
\item For each $n>0$, $n$-ary function symbol $F$ of $S$ and n-tuple $\bar{a} \in \mathcal{A}$, $f(F^{\mathcal{A}}(\bar{a}))=F^{\mathcal{B}}(f\bar{a})$.
\end{enumerate}
Where $\bar{a}=(a_{0},...,a_{n-1})$ and $f\bar{a}=(fa_{0},...,fa_{n-1})$.
\end{Definition}
The conditions that $f$ satisfies in order to be an isomorphism are now given,
\begin{Definition}
An embedding of $\mathcal{A}$ into $\mathcal{B}$ is a homomorphism $f: \mathcal{A} \rightarrow \mathcal{B}$ which is injective and satisfies,
\begin{enumerate}
\item For each $n>0$, each n-ary relation symbol $R$ of $S$ and each n-tuple $\bar{a} \in \mathcal{A}$, $\bar{a} \in R^\mathcal{A} \Leftrightarrow f\bar{a} \in R^\mathcal{B}$.
\end{enumerate}
Furthermore, $A$ and $B$ are isomorphic, written $A \cong B$, when there exists a surjective embedding $f: \mathcal{A} \rightarrow \mathcal{B}$.
\end{Definition}
The second question is recapitulated formally as motivation for classifying the properties of standard and non-standard models of Peano arithmetic.
\begin{Problem}
Given the standard model $\mathcal{N}$ of $\mathsf{PA}$, if $\mathcal{M} \vDash \mathsf{PA}$ and $\mathcal{M} \ncong \mathcal{N}$, then how many such countable models $\mathcal{M}$ are there and how do they differ from $\mathcal{N}$?
\end{Problem}
The answer to this question requires the theory of non-standard models and will be answered by Theorem 5 and Theorem 8 in the forthcoming sections.

\section{Non-standard Models of $\mathsf{PA}$}
This section presents the ancillary terminology and results necessary to prove the existence of non-standard models of $\mathsf{PA}$ and that there are $2^{\aleph_{0}}$ such countable models. If we are to consider $\mathcal{N}$ as the standard model, then it seems natural to define a non-standard model $\mathcal{M}$ of $\mathsf{PA}$ such that $\mathcal{N} \ncong \mathcal{M}$.
\begin{Definition}
A non-standard model $\mathcal{M}$ of $\mathsf{PA}$ is an $\mathsf{L_{A}}$-structure such that $\mathcal{M} \vDash \varphi$, for all $\varphi \in \mathsf{PA}$, and $\mathcal{M} \ncong \mathcal{N}$, where $\mathcal{N}$ is the standard model of $\mathsf{PA}$.
\end{Definition}
That is to say, a model $\mathcal{M}$ of $\mathsf{PA}$ is non-standard when there does not exist a surjective embedding $f:\mathcal{N} \rightarrow \mathcal{M}$. Unpacking the definitions, this means that for any homomorphism $f:\mathcal{N}\rightarrow \mathcal{M}$, either there exists a constant symbol, relation, or function which is not mapped to ($f$ is not a bijection), or the condition for an embedding - that $\overline{a} \in R^{\mathcal{N}} \Leftrightarrow \overline{a} \in R^{\mathcal{M}}$ - does not hold. The explicit construction of such a non-standard model $\mathcal{M}$ of $\mathsf{PA}$ will require a connection between the satisfiability of a theory and some new constant symbol which ensures that $\mathcal{M} \ncong \mathcal{N}$. This connection is forged by the use of G\"{o}del's Completeness Theorem and the Compactness Theorem, two frequently used theorems in model theory and mathematical logic which were initially proven by Kurt G\"{o}del in 1929 and 1930, respectively.\cite{Marker}
\begin{Theorem}
(G\"{o}del's Completeness Theorem): Let $T$ be an $\mathsf{L}$-theory where $\varphi$ is an $\mathsf{L}$-sentence. Then $T \vDash \varphi$ if and only if $T \vdash \varphi$.
\end{Theorem}
G\"{o}del's Completeness theorem relates the semantic notion of a model's interpretation making a sentence true (i.e., that a sentence is a logical consequence), to the syntactic notion of proving a sentence. The following corollary connects the Completeness Theorem to the Compactness Theorem in a way that is essential to the existence proof of non-standard models of $\mathsf{PA}$.\cite{Marker}
\begin{Corollary}
$T$ is consistent if and only if $T$ is satisfiable.
\end{Corollary}
\begin{proof}
Assume to the contrary that there exists a theory $T$ such that $T$ is consistent and $T$ is not satisfiable. Since $T$ is not satisfiable, there does not exist a model $\mathcal{M}$ of $T$. So, any $\mathsf{L}$-structure which attempts to model $T$ is a model of $\perp$. Then, $T \vDash \perp$ so by the Completeness Theorem, $T \vdash \perp$; yet this contradicts the assumption that $T$ is consistent. Assume to the contrary that there exists a theory $T$ such that $T$ is satisfiable and $T$ is not consistent; this is an immediate contradiction by the definition of satisfiability. Therefore, $T$ is consistent if and only if $T$ is satisfiable.
\end{proof}
\begin{Theorem}(Compactness Theorem):
$T$ is satisfiable if and only if every finite subset of $T$ is satisfiable.
\end{Theorem}
Given the Compactness Theorem, the existence of non-standard models was initially proved by the logician and mathematician Thoralf Skolem in 1934. In a modern form, following the general outline as presented in \cite{Kaye}, the proof is much simpler than was originally given by Skolem.
\begin{Theorem}
There exists non-standard models of $\mathsf{PA}$.
\end{Theorem}
\begin{proof}
We want to prove that there exists a model $\mathcal{M}$ for $\mathsf{PA}$ which is not isomorphic to the standard model $\mathcal{N}$. Let $\overline{n}$ be the value of the numeral ${n}$ formed by $\overline{n} = \underbrace{1+ ... + 1}_\text{n 1's}$, and let $c$ be a new constant symbol. Then let,
\begin{equation*}
T_{k} = \{\mathsf{Ax_{PA}}\} \cup \{\neg(c=\overline{n}) | \text{ } \overline{n} < k\}
\end{equation*}
be a set of axioms in the language $\mathsf{L_{A}} \cup \{c\}$, where $n,k \in \mathbb{N}$. For a given $k$, give the interpretation $c^\mathcal{N} = \overline{k}^\mathcal{N}$. Then, since $\mathsf{PA}$ is consistent,\footnote{Gentzen's consistency proof (1936) for $\mathsf{PA}$ uses transfinite ordinal induction up to $\epsilon_{0}=\text{sup}\{\omega,\omega^\omega,\omega^{\omega^\omega},...\}$.\cite{Smith}} $T_{k}$ is consistent, and thus satisfiable by Corollary 1, for each $k \in \mathbb{N}$. Therefore, the standard model of Peano arithmetic is a model for $T_{k}$; that is to say, $\mathcal{N} \vDash T_{k}$. Since $T_{1} \subseteq T_{2} \subseteq \cdot\cdot\cdot$, and each $T_{i}$ is satisfiable for $i \in \mathbb{N}$, $$T_{\omega} = \bigcup_{i\in \mathbb{N}} T_{i}$$ is satisfiable by the Compactness Theorem. So, there exists an $(\mathsf{L_{A}} \cup \{c\})$-structure $\mathcal{M} = \langle \mathbb{N},+,\times,\leq,0,1,c \rangle$ such that $\mathcal{M} \vDash T_{\omega}$, and thus $\mathcal{M} \vDash \mathsf{PA}$.

Assume to the contrary that $\mathcal{M} \cong \mathcal{N}$, then there exists a surjective embedding $f: \mathcal{N} \rightarrow \mathcal{M}$ and so $f(n^{\mathcal{N}})=n^{\mathcal{M}}$, for all $n \in \mathbb{N}$. But since $c \neq n$, for all $n \in \mathbb{N}$, there does not exist an image in $\mathcal{M}$ of $c$ under $f$, which contradicts that $f$ is a surjective embedding. Therefore, $\mathcal{M}$ is a model for Peano arithmetic and $\mathcal{M} \ncong \mathcal{N}$, so $\mathcal{M}$ is non-standard.
\end{proof}

\begin{figure}[H]
\begin{center}
\includegraphics[scale=0.3]{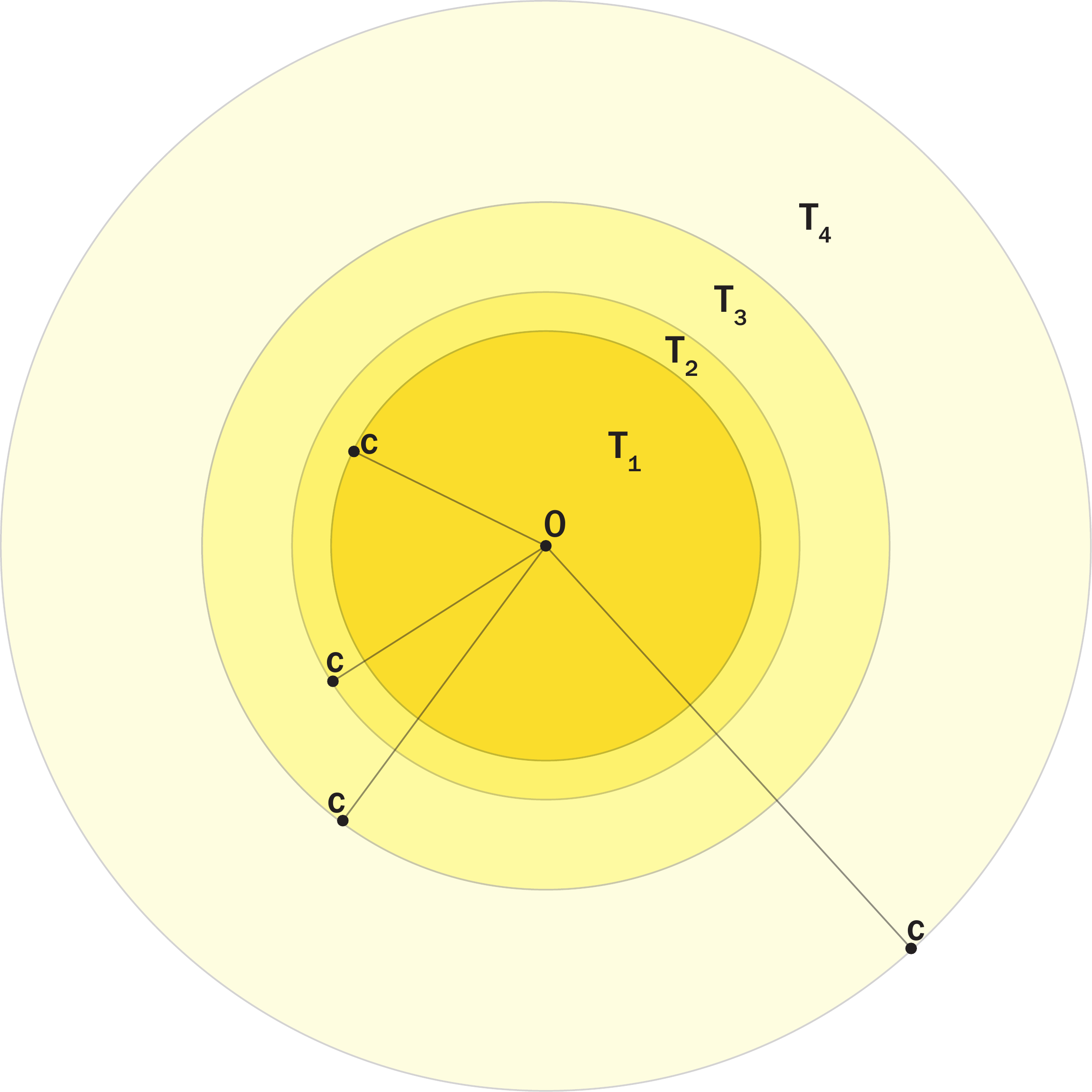}
\end{center}
\caption{Representation of $T_{1},...,T_{4}$ where in each case $c \neq n$ for all $n<k$.}
\end{figure}

\pagebreak

The existence of non-standard models of Peano arithmetic has then been established through invoking the Completeness Theorem and the Compactness Theorem, two cornerstones of mathematical logic and model theory. With some additional terminology and the use of the G\"{o}del-Rosser Theorem in the context of $\mathsf{PA}$, the first part of Problem 1 raised at the end of the Lexicon can be answered by proving that there are $2^{\aleph_{0}}$ non-standard countable models of $\mathsf{PA}$.

The first new notion to introduce is that of a substructure, or equivalently, an extension of an $\mathsf{L}$-structure, which relates two $\mathsf{L}$-structures by the cardinality of the domain and a map between the domains called an inclusion map.\cite{Marker}
\begin{Definition}
Let $\mathcal{M}$ and $\mathcal{K}$ be two $\mathsf{L}$-structures where $\text{dom}(\mathcal{M}) \subseteq \text{dom}(\mathcal{K})$. If the inclusion map $i: \text{dom}(\mathcal{M}) \hookrightarrow \text{dom}(\mathcal{K})$, which maps every element of $\text{dom}(\mathcal{M})$ to some element in $\text{dom}(\mathcal{K})$, is an embedding, then $\mathcal{M}$ is a substructure of $\mathcal{K}$, or equivalently, $\mathcal{K}$ is an extension of $\mathcal{M}$, written $\mathcal{M} \prec \mathcal{K}$.
\end{Definition}
Intuitively, the idea of an extension of an $\mathsf{L}$-structure is similar to the notion of a superset, but requires that there is a map $i: \text{dom}(\mathcal{M}) \hookrightarrow \text{dom}(\mathcal{K})$ that satisfies the condition of an embedding given in the Lexicon. Visually,
\begin{figure}[H]
\begin{center}
\includegraphics[scale=0.2]{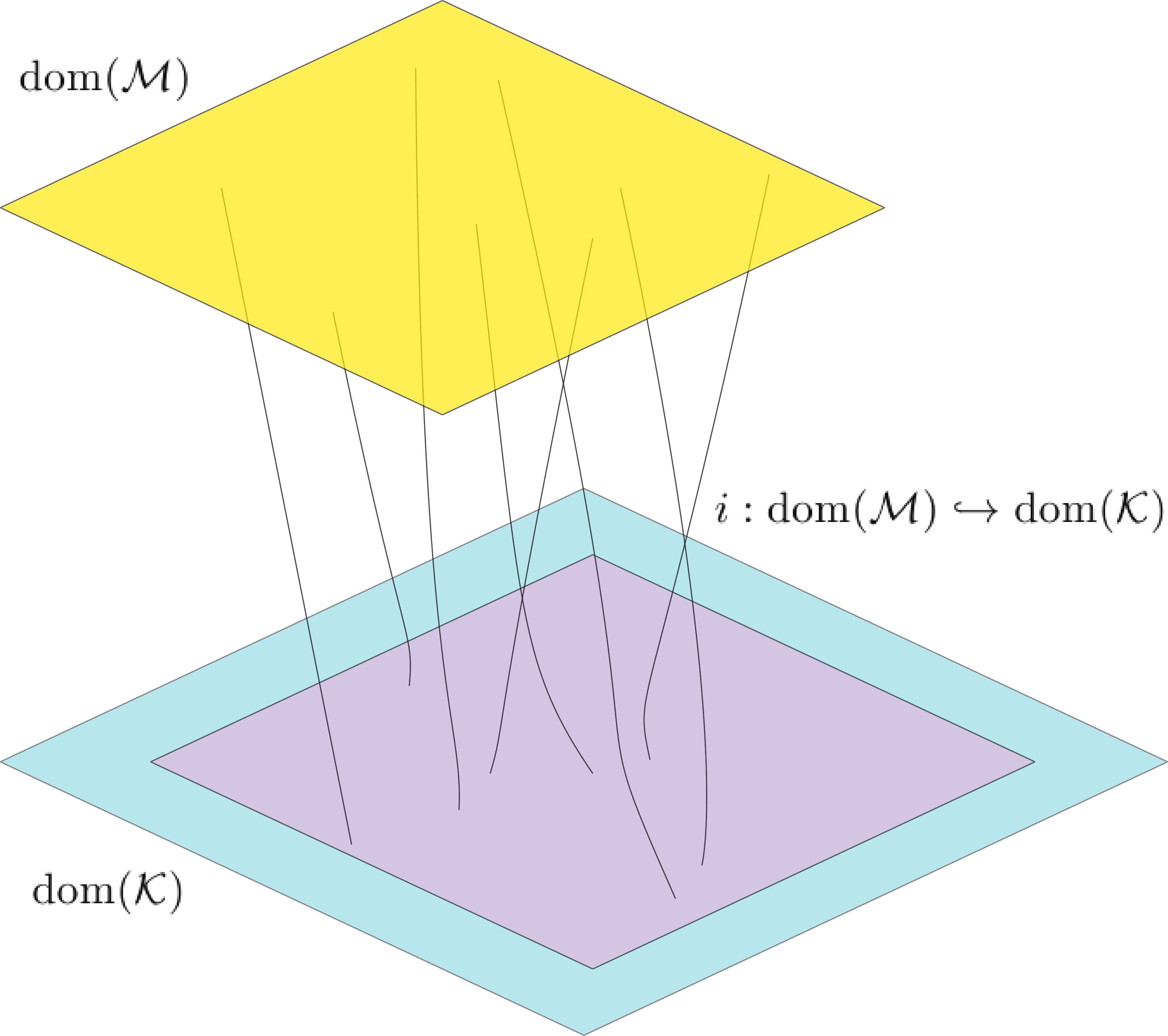}
\end{center}
\caption{An inclusion map from the substructure $\mathcal{M}$ to the extension $\mathcal{K}$}
\end{figure}
The use of extensions of $\mathsf{L}$-structures, and in particular, extensions of models, is an integral component of showing that there are continuum-many, that is $2^{\aleph_{0}}$, distinct complete extensions (and thus non-isomorphic models) of $\mathsf{PA}$. In order to accomplish this task without a lengthy digression, the G\"{o}del-Rosser Theorem is stated without proof, where $\mathsf{PA^-}$ is $\mathsf{PA}$ without the induction schema.\cite{Kaye}
\begin{Theorem}
(G\"{o}del-Rosser Theorem): Let $T$ be a recursively axiomatized $\mathsf{L_{A}}$-theory that extends $\mathsf{PA^-}$. Then there is a $\Pi_{1}$ sentence $\psi$ such that $T \nvdash \psi$ and $T\nvdash \neg \psi$.
\end{Theorem}
The following theorem provides an answer to the first part of Problem 1.
\begin{Theorem}
There are $2^{\aleph_{0}}$ non-standard countable models of $\mathsf{PA}$.
\end{Theorem}
\begin{proof}
Let $\psi$ be the G\"{o}del-Rosser sentence of $\mathsf{PA}$ and consider the two complete extensions to $\mathsf{PA} \cup \psi$ and $\mathsf{PA} \cup \neg \psi$ which are both consistent by the G\"{o}del-Rosser Theorem and furthermore satisfiable by Corollary 1. Hence there exists models $\mathcal{M}_{0} \vDash \mathsf{PA} \cup \psi$ and $\mathcal{M}_{1} \vDash \mathsf{PA} \cup \neg \psi$. Proceeding by induction, let $\psi'$ be the G\"{o}del-Rosser sentence of $\mathsf{PA} \cup \psi$ and let $\psi''$ be the G\"{o}del-Rosser sentence of $\mathsf{PA} \cup \neg \psi$; then there exists models $\mathcal{M}_{0,0} \vDash (\mathsf{PA} \cup \psi) \cup \psi'$ and $\mathcal{M}_{0,1} \vDash (\mathsf{PA} \cup \psi) \cup \neg \psi'$. Similarly, there exists models $\mathcal{M}_{1,0} \vDash (\mathsf{PA} \cup \neg \psi) \cup \psi''$ and $\mathcal{M}_{1,1} \vDash (\mathsf{PA} \cup \neg \psi) \cup \neg \psi''$. Therefore, for the standard model $\mathcal{N} \vDash \mathsf{PA}$ we have that $\mathcal{N} \prec \mathcal{M}_{0}$ and $\mathcal{N} \prec \mathcal{M}_{1}$. Continuing on we have that $\mathcal{M}_{0} \prec \mathcal{M}_{0,0}$ and $\mathcal{M}_{0} \prec \mathcal{M}_{0,1}$ while $\mathcal{M}_{1} \prec \mathcal{M}_{1,0}$ and $\mathcal{M}_{1} \prec \mathcal{M}_{1,1}$. Continue this process for $\mathcal{M}_{0,0}, \mathcal{M}_{0,1}, \mathcal{M}_{1,0}$, and $\mathcal{M}_{1,1}$ letting 0 be the subscript in place of $x$ for $\mathcal{M}_{i,...,j,x}$ when $\mathcal{M}_{i,...,j} \vDash \Psi$ is a substructure of the model $\mathcal{M}_{i...,j,x} \vDash \Psi \cup \psi$ and letting 1 be the subscript in place of $x$ for $\mathcal{M}_{i,...,j,x}$ when $\mathcal{M}_{i,...,j} \vDash \Psi$ is a substructure of the model $\mathcal{M}_{i,...,j,x} \vDash \Psi \cup \neg \psi$. Since the cardinality of the set of all functions $f: \mathbb{N} \rightarrow \{0,1\}$ is $2^{\aleph_{0}}$, that is, $|\{0,1\}^{\mathbb{N}}|=2^{\aleph_{0}}$, it is clear that there exists $2^{\aleph_{0}}$ models of extensions of $\mathsf{PA}$ by corresponding a function $f \in \{0,1\}^{\mathbb{N}}$ to the index of a model $\mathcal{M}_{i}$ (which is a set of $0$'s and $1$'s). Let $\mathcal{M}_{i}$ and $\mathcal{M}_{j}$ be two distinct models (where $i$ and $j$ are the indices of $0$'s and $1$'s separated by commas) of $\Psi$ and $\Phi$; two eventual theories generated by taking finitely many iterations of the union of $\psi$ or $\neg \psi$ for the previous theory. Let $\mathcal{K}$ be the model for which $\mathcal{K} \prec \cdot \cdot \cdot \prec \mathcal{M}_{i}$ and $\mathcal{K} \prec \cdot\cdot\cdot \prec \mathcal{M}_{j}$. Then, since $i \neq j$, $\mathcal{M}_{i}$ and $\mathcal{M}_{j}$ differ in index from at least the model $\mathcal{K}$ and thus $\mathcal{M}_{i} \ncong \mathcal{M}_{j}$. Hence $\mathcal{M}_{k} \ncong \mathcal{N}$ for each $\mathcal{M}_{k}$ which is an extension of $\mathcal{N}$. Therefore, since for any $k$, $\mathcal{M}_{k} \vDash \mathsf{PA}$, there exists $2^{\aleph_{0}}$ non-standard countable models of $\mathsf{PA}$.
\end{proof}
\begin{figure}[H]
\begin{center}
\includegraphics[scale=0.49]{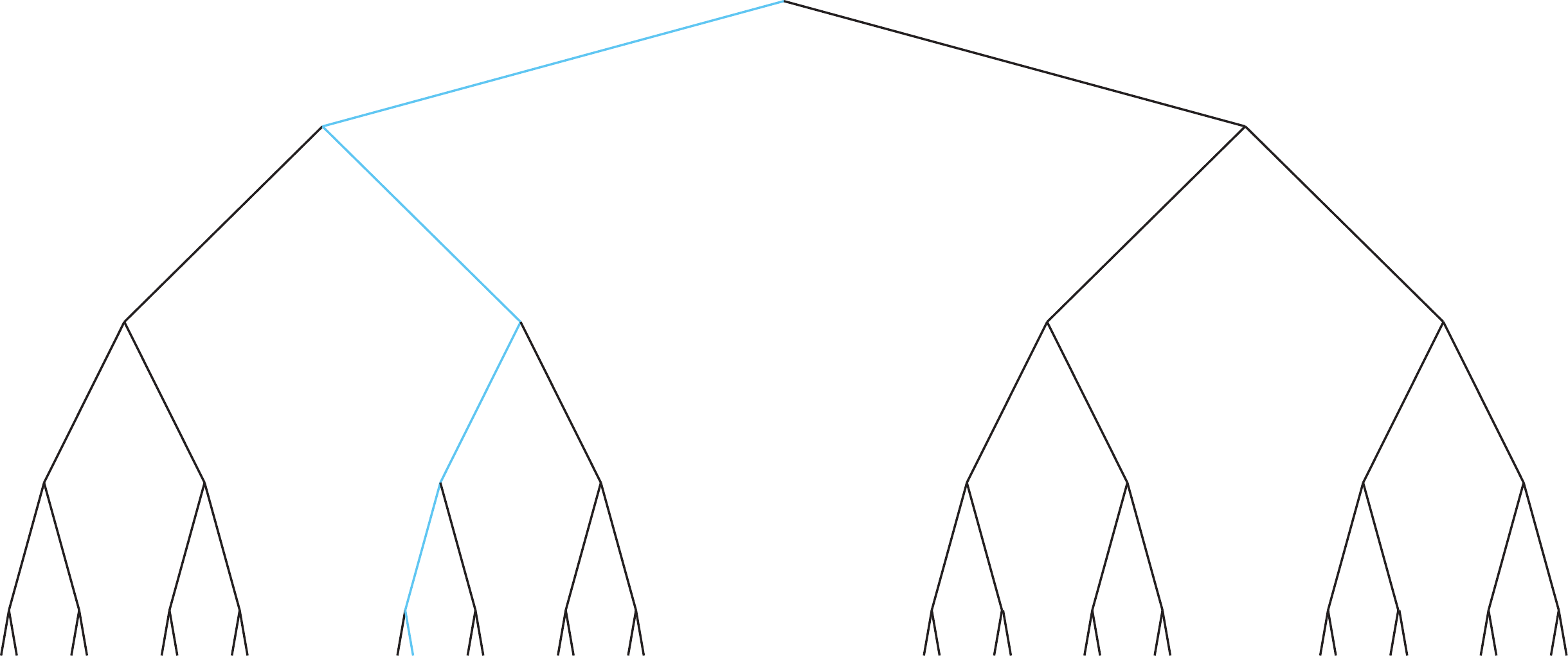}
\end{center}
\caption{Representation of model extensions as a binary $\omega$-tree.}
\end{figure}
Seen visually, the example blue branch of the binary $\omega$-tree represents the chain of extensions $\mathcal{N} \prec \mathcal{M}_{0}\prec \mathcal{M}_{0,1} \prec \mathcal{M}_{0,1,0} \prec \mathcal{M}_{0,1,0,0} \prec \mathcal{M}_{0,1,0,0,1}$.

\section{Overspill and Coding Sets}
The present analysis now turns to answer the second part of Problem 1, which asks how models isomorphic to the standard model and non-standard models of $\mathsf{PA}$ differ. Thus far, the primary concern of how to classify models has been by isomorphism and extension, noting the powerful tool that when a model $\mathcal{M}$ is an extension of $\mathcal{K}$, or that $\mathcal{K}$ is a substructure of $\mathcal{M}$, it is immediately the case that $\mathcal{M} \ncong \mathcal{K}$ (as was used heavily in the proof that there exists $2^{\aleph_{0}}$ non-standard countable models of $\mathsf{PA}$). The second main classifying method that will be used to tell how non-standard models of $\mathsf{PA}$ differ from models isomorphic to the standard model is based off of the properties that the functions found in the signature of the model have. That is to say, given a non-standard model $\mathcal{K}$ for $\mathsf{PA}$ with a signature $\langle \text{dom}(\mathcal{K}),\oplus,\otimes,\leq,0,1,... \rangle$, the functions $\oplus$ and $\otimes$ may or may not be recursive. From this, the model itself is considered recursive or not depending on if it has recursive functions. Formally, as mentioned in \cite{Kaye},
\begin{Definition}
An $\mathsf{L_A}$-structure $\mathcal{M}$ is recursive if and only if there are recursive functions $\oplus: \mathbb{N}^2 \rightarrow \mathbb{N}$, $\otimes: \mathbb{N}^2 \rightarrow \mathbb{N}$, a binary recursive relation $< \subseteq \mathbb{N}^2$ and natural numbers $n_{0},n_{1} \in \mathbb{N}$ such that $\langle \mathbb{N},\oplus,\otimes,<,n_{0},n_{1} \rangle \cong \mathcal{M}$.
\end{Definition}
The strategy for proving Tennenbaum's Theorem is to assume that there exists a recursive non-standard model $\mathcal{K}$ of $\mathsf{PA}$ and then show that this leads to a contradiction. This may be achieved by constructing a non-recursive set $\chi$ and then showing - under the assumption that $\mathcal{K}$ exists - that $\chi$ is recursive, leading to a contradiction.  Such a non-recursive set $\chi$ must involve non-standard elements and be constructible for any non-standard model of $\mathsf{PA}$, so the two notions of overspill and set coding are required to fulfill these requirements.

\textit{Firstly}, overspill essentially states that in a non-standard model $\mathcal{M}$, if all of the standard elements $n$ of $\mathcal{M}$ satisfy $\varphi(n)$, then a non-standard element does as well. Overspill will be used to introduce non-standard elements into sentences which have an arbitrary standard element $n$ ranged over. That is to say, given some sentence $\psi$ which holds for any $n$, so $\mathsf{PA} \vdash \psi(x,n)$, then $\mathcal{M} \vDash \psi(x,e)$ by overspill. The proof follows simply by the induction axiom of $\mathsf{PA}$. \cite{Kossak}
\begin{Theorem}(Overspill):
Let $\mathcal{M}$ be a non-standard model of $\mathsf{PA}$. Then, for some $\varphi \in \mathsf{PA}$, if $\mathcal{M} \vDash \varphi(n)$ for all standard elements $n$, then there exists a non-standard element $e$ such that $\mathcal{M} \vDash \varphi(e)$.
\end{Theorem}
\begin{proof}
Assume to the contrary that there exists a model $\mathcal{M}$ of $\mathsf{PA}$ for which $\mathcal{M} \vDash \varphi(n)$, for all standard elements $n$, and $\mathcal{M} \nvDash \varphi(e)$, for all non-standard elements $e$. By assumption, $\mathcal{M} \vDash \varphi(0)$ and $\mathcal{M} \vDash \forall x(\varphi(x) \Rightarrow \varphi(x+1))$, since $\mathcal{M} \vDash \varphi(n)$ for all standard elements $n$. Then, by the induction axiom of $\mathsf{PA}$, $\mathcal{M} \vDash \forall x \varphi(x)$, and so $\mathcal{M} \vDash \varphi(e)$, which contradicts the assumption.
\end{proof}

\textit{Secondly}, the process of coding a set will bridge the gap between non-standard models of $\mathsf{PA}$ and non-recursive sets by an application of overspill. Given a non-standard model $\mathcal{M} \vDash \mathsf{PA}$, the process of describing a non-recursive set in the domain of $\mathcal{M}$ is accomplished by using a prime number coding technique for sets. Specifically, given the abbreviation that $\pi_{n}$ is the $n^{th}$ prime number, the set of numbers $n$ for which $\mathcal{M} \vDash \exists k(c = k \times \pi_{n})$ will uniquely determine a set of numbers. With G\"{o}del numbering and other coding techniques, prime numbers are essential for ensuring a desired form of uniqueness; in the case of the arithmetization of syntax \cite{Smith}, G\"{o}del numbering relies on the fundamental theorem of arithmetic for a unique decoding. The uniqueness desired for coding sets is that a single element of the $\text{dom}(\mathcal{M})$ uniquely determines a set of numbers. This is exactly achieved by $c$ by the formula $\varphi(n,c)=\exists k(c=k \times \pi_{n})$. The standard system of sets coded in a model refers to the set of all such sets which are coded by some element $c$ in the model.\cite{Kennedy}

\begin{Definition}
For a non-standard model $\mathcal{M} \vDash \mathsf{PA}$, the standard system of sets coded in $\mathcal{M}$, denoted $\mathsf{SSy}(\mathcal{M})$, is the set of all $A \subseteq \mathbb{N}$ such that $A=\{n \in \mathbb{N}:\mathcal{M} \vDash \varphi(n,c)\}$, for $\varphi \in \mathsf{PA}$ and some $c \in \mathcal{M}$.
\end{Definition}

Given the method for coding a set by one element of $\text{dom}(\mathcal{M})$, the notion of recursive inseparability allows for an application of set coding by placing a restriction on pairwise disjoint sets as follows.\cite{Kaye}

\begin{Definition}
Disjoint sets $A,B \subseteq \mathbb{N}$ are recursively inseparable if and only if there does not exist recursive $C \subseteq \mathbb{N}$ such that $A \subseteq C$ and $C \cap B = \varnothing$.
\end{Definition}
Given that recursive inseparability gives a criteria for the existence of a non-recursive set, the description of a non-recursive set coded in any non-standard model $\mathcal{M} \vDash \mathsf{PA}$ is on the brink of derivation. The following proof fills in the details sketched in \cite{Kennedy}.

\begin{Theorem}
Let $\mathcal{M}$ be a non-standard model of $\mathsf{PA}$. Then $\mathsf{SSy}(\mathcal{M})$ contains a non-recursive set.
\end{Theorem}
\begin{proof}
Let $A,B \subseteq \mathbb{N}$ be recursively enumerable sets that are recursively inseparable. Let $f$ and $g$ be the $\mu$-recursive functions which enumerates $A$ and $B$, respectively. Since every $\mu$-recursive function is $\Sigma_{1}$, there exists $\Sigma_{1}$ formulas which express $A$ and $B$ of the form $\exists y \alpha(x,y)$ and $\exists z \beta(x,z)$, respectively, with $\Delta_{0}$ kernels $\alpha$ and $\beta$. Since $A, B \subseteq \mathbb{N}$ and $\mathcal{N} \prec \mathcal{M}$, there exists an inclusion map embedding $i: \mathbb{N} \hookrightarrow \text{dom}(\mathcal{M})$. Since every embedding is a homomorphism, then $f(i^{\mathcal{N}}(x))=i^{\mathcal{M}}(f(x))$ and $g(i^{\mathcal{N}}(x))=i^{\mathcal{M}}(g(x))$. That is to say, since the interpretations of $f$ and $g$ are preserved during a model extension, the same $\Sigma_{1}$ formulas will express $A$ and $B$ in the extension $\mathcal{M}$. Therefore, $A \cap B = \varnothing$ since $A$ and $B$ are recursively inseparable, and so for each $k \in \mathbb{N}$,
$$\mathcal{M} \vDash (\forall x < k)(\forall y < k)(\forall z < k) \neg(\alpha(x,y) \wedge \beta(x,z))$$
Thus, by overspill there exists a non-standard element $e$ such that,
$$\mathcal{M} \vDash (\forall x < e)(\forall y < e)(\forall z < e) \neg(\alpha(x,y) \wedge \beta(x,z))$$
Let $C \subseteq \mathbb{N}$ such that $C = \{n \in \mathbb{N}: \mathcal{M} \vDash (\exists y < e) \alpha(n,y)\}$. Then since $e$ is non-standard and the inclusion map embedding $i$ guarantees that the $\Sigma_{1}$ formulas expressing $f$ is preserved, $A \subseteq C$ and $C \cap B = \varnothing$. Therefore, $C$ is non-recursive since $A$ and $B$ are recursively inseparable.
\end{proof}

\section{Tennenbaum's Theorem}
With the terminology and results presented in the previous section in place, Tennenbaum's Theorem can now be proved following the outline from \cite{Kennedy}.
\begin{Theorem}(Tennenbaum's Theorem): If $\mathcal{M} = \langle \mathcal{M},\oplus,\otimes,<,0,1 \rangle$ is a countable model of $\mathsf{PA}$ such that $\mathcal{M} \ncong \mathcal{N} = \langle \mathbb{N},+,\times,<0,1 \rangle$, then $\mathcal{M}$ is not recursive.
\end{Theorem}
\begin{proof}
Let $\mathcal{M}$ be a non-standard model for $\mathsf{PA}$ and let $C \in \mathsf{SSy}(\mathcal{M})$ be non-recursive. Then, there exists a $c \in \mathcal{M}$ such that $C = \{n \in \mathbb{N}:\mathcal{M} \vDash \varphi(n,c)\}$, where $\varphi(n,c) = \exists y(c= y \otimes \pi_{n})$ uniquely codes $C$ by the element $c$. Assume to the contrary that the function $\oplus: \mathbb{N}^2 \rightarrow \mathbb{N}$ is recursive. Then, let
\begin{equation*}
\psi(n,c) =(c = \underbrace{y \oplus \cdot \cdot \cdot \oplus y}_\text{$\pi_{n}$ y's}) \vee (c = \underbrace{y \oplus \cdot \cdot \cdot \oplus y}_\text{$\pi_{n}$ y's} \oplus 1) \vee \cdot \cdot \cdot \vee (c = \underbrace{y \oplus \cdot \cdot \cdot \oplus y}_\text{$\pi_{n}$ y's} \underbrace{\oplus 1 \oplus \cdot \cdot \cdot \oplus 1}_\text{$\pi_{n}-1$ 1's})
\end{equation*}
By Euclidean Division\footnote{Since $\mathcal{M} \vDash \mathsf{PA}$ and $\mathsf{PA}$ proves Euclidean division, then for $c \in \mathcal{M}$, and $n,y,r \in \mathbb{N}$ with $c,y \neq 0$, $\exists ! y \exists !r$ such that $c = (y \otimes n) \oplus r$, where $0 \leq r < n$.}, $\exists!\pi_{n}^{\mathcal{M}} \exists!r \in \mathbb{N}$ such that $c = (y \otimes \pi_{n}^{\mathcal{M}}) \oplus r$ and $0 \leq r < \pi_{n}^{\mathcal{M}}$. If $r = 0$, then $n \in C$ since the disjunct $c = \underbrace{y \oplus \cdot \cdot \cdot \oplus y}_\text{$\pi_{n}$ y's}$ of $\psi(n,c)$ is true and so $\varphi(n,c)$ is true. If $r \neq 0$, then $n \notin C$ since one of the other disjuncts in $\psi(n,c)$ is true and so $\varphi(n,c)$ is not true. Hence, $C$ is recursive since it is correctly decidable in $\mathcal{M}$ if $n \in C$ or $n \notin C$; yet this contradicts the fact that $C$ is non-recursive. Therefore, $\mathcal{M}$ is not recursive since $\oplus: \mathbb{N}^2 \rightarrow \mathbb{N}$ is not recursive.
\end{proof}
With Tennenbaum's Theorem, an answer can be given to Problem 1 which stated ``How many countable non-standard models $\mathcal{M}\vDash \mathsf{PA}$ are there and how do they differ from $\mathcal{N}$?''. By Theorem 5, there are $2^{\aleph_{0}}$ countable non-standard models of $\mathsf{PA}$, and by Tennenbaum's Theorem every countable non-standard model of $\mathsf{PA}$ is not recursive. Thus, an explicit line is drawn between the standard model of $\mathsf{PA}$ and all of the countable non-standard models of $\mathsf{PA}$, as any attempts to formulate arithmetical operations in such a non-standard model will be severely restricted.

The preceding paper has arrived at three primary and interdependent conclusions which may be summarized as follows. Firstly, through invoking G\"{o}del's Completeness Theorem and the Compactness Theorem, an existence proof of non-standard models of Peano arithmetic was given. Secondly, by appealing to the G\"{o}del-Rosser Theorem and the notion of a model extension, it was proved that there are $2^{\aleph_{0}}$ countable non-standard models of Peano arithmetic. Thirdly, Tennenbaum's Theorem was proved by overspill, set coding, and the existence of a non-recursive set in every non-standard model of Peano arithmetic. Taken cumulatively, the methods and results required to prove the presented theorems were important for $20^{th}$ century mathematicians and logicians concerned with model theory and mathematical logic. The notable research programs on such topics conducted by Skolem, Rosser, Tarski, Mostowski and others, demonstrate the undeniable import of Tennebaum's Theorem.\cite{Kennedy}

\pagebreak

\end{document}